\documentclass[12pt, a4paper]{article}
\usepackage{titlesec} 
\usepackage{amsfonts}
\usepackage{amsthm}
\usepackage{amssymb}
\usepackage{amsmath}
\usepackage{theoremref} 
\usepackage{enumerate}
\usepackage{bbm}
\usepackage{authblk}
\usepackage{commath}
\usepackage{mathtools}
\usepackage{tikz}
\usepackage{xparse}
\usetikzlibrary{backgrounds,patterns,calc, arrows}

\makeatletter
\g@addto@macro\bfseries{\boldmath}
\makeatother

\newcommand{\A}{\mathcal{A}}
\newcommand{\C}{\mathbb{C}}

\newcommand{\T}{\mathbb{T}}
\newcommand{\z}{\zeta}
\newcommand{\conj}[1]{\overline{#1}}
\newcommand{\D}{\mathbb{D}}
\newcommand{\B}{\mathcal{B}}

\newcommand{\m}{\textit{m}}

\newcommand{\BMOA}{\textbf{BMOA}}
\newcommand{\BMO}{\textbf{BMO}}

\renewcommand\Re{\operatorname{Re}}
\renewcommand\Im{\operatorname{Im}}

\newtheorem{thm}{Theorem}[section]
\newtheorem{thmletter}{Theorem}
\newtheorem{corletter}{Corollary}
\newtheorem{lemma}[thm]{Lemma}

\theoremstyle{definition}

\theoremstyle{definition}

\newtheorem*{problem*}{Problem}
\newtheorem*{Conj*}{Conjecture}

\titleformat{\section}
{\normalfont\scshape\centering}{\thesection}{1em}{}

\titleformat{\subsection}
[runin]{\bfseries\scshape}{\thesubsection}{0.5em}{}

\titleformat{\subsubsection}
[runin]{}{\thesubsubsection}{0.5em}{}

\begin{document}
\title{\textbf{Generalized Ces\`aro operators: geometry of spectra and quasi-nilpotency}}
\author{Adem Limani, Bartosz Malman}
\date{}
\maketitle

\begin{abstract}
\noindent For the class of Hardy spaces and standard weighted Bergman spaces of the unit disk we prove that the spectrum of a generalized Ces\`aro operator $T_g$ is unchanged if the symbol $g$ is perturbed to $g+h$ by an analytic function $h$ inducing a quasi-nilpotent operator $T_h$, i.e. spectrum of $T_h$ equals $\{0\}$. We also show that any $T_g$ operator which can be approximated in the operator norm by an operator $T_h$ with bounded symbol $h$ is quasi-nilpotent. In the converse direction, we establish an equivalent condition for the function $g \in \BMOA$ to be in the $\BMOA$-norm closure of $H^{\infty}$. This condition turns out to be equivalent to quasi-nilpotency of the operator $T_g$ on the Hardy spaces. This raises the question whether similar statement is true in the context of Bergman spaces and the Bloch space. Furthermore, we provide some general geometric properties of the spectrum of $T_{g}$ operators.
\end{abstract}

\section{Introduction}

Let $\D$ denote the unit disk of the complex plane $\C$. The generalized Ces\`aro operator with symbol $g: \D \to \C$ acts on an analytic function $f:\D \to \C$ by $$T_gf(z) = \int_0^z f(\z) g'(\z) \,d\z, \quad z \in \D.$$ We will for the most part be working in the context of $T_g$ acting on the Hardy spaces and the standard weighted Bergman spaces. It is a classical result in the theory that $T_g$ is bounded on the Hardy spaces $H^p,$ for $0 < p < \infty,$ if and only if the symbol $g$ lies in the space $\BMOA$ (see \cite{alemansiskakishardy}). The Hardy spaces $H^p$ are defined, as usual, to be the spaces of analytic functions in $\D$ which satisfy $$\|f\|_{H^p}^p := \sup_{0 < r < 1} \int_\T |f(r\z)|^p \,d\m(\z) = \int_\T |f(\z)|^p d\m(\z) < \infty,$$ where $d\m$ denotes the normalized Lebesgue measure on the unit circle $\T = \{ z : |z| = 1 \}$ and the last integral is defined by the boundary values of $f$. The space $\BMOA$ is the dual of $H^1$ under the usual Cauchy pairing \begin{equation*} \lim_{r \to 1^-} \int_\T f(\z)\conj{g(r\z)} d\m(\z), \quad f \in H^1, g \in \BMOA.
\end{equation*} The weighted Bergman space $L^{p,\alpha}_a$, for $0 < p < \infty$ and $\alpha > -1$, consists of analytic functions in $\D$ which satisfy $$\|f\|_{L^{p,\alpha}_a}^p := \int_{\D} |f(z)|^p (1-|z|^2)^\alpha dA(z) < \infty,$$ where $dA$ denotes the normalized area measure. In this case, the criterion for boundedness of $T_g$ is that the symbol $g$ belongs to the Bloch space $\B$ (see, for instance, \cite{acs}), which is the space of analytic functions in $\D$ satisfying $$\sup_{z \in \D} (1-|z|^2)|g'(z)| < \infty.$$  

The purpose of the present paper is to shed light on some spectral properties of $T_g$ operators, with special emphasis on the quasi-nilpotent operators.

We start by briefly describing a common approach that has been followed with the purpose of characterizing the spectrum of $T_g$ operators for several classes of Banach spaces $X$ of analytic functions. We will always implicitly assume that the evaluations $f \mapsto f(\lambda)$ are bounded on any space considered. Since $T_gf(0) = 0$ for any $g$ and $f$, if $X$ contains functions which do not vanish at $z = 0$, then $T_g$ is not surjective, and thus the point $0$ is always contained in the spectrum. It is also easy to see that the operator $T_g$ has no eigenvalues, and that for any $\lambda \in \C \setminus \{0\}$ and any analytic $h$, the equation $(I - \lambda^{-1}T_g)f = h$ has the unique solution 
\begin{equation} \label{resolventeq} f(z) = R_g(\lambda) h(z) := h(0)e^{\frac{g(z)}{\lambda}} + e^{\frac{g(z)}{\lambda}} \int_0^z e^{-\frac{g(\z)}{\lambda}} h'(\z) \,d\z.\end{equation} Therefore, the operator $T_g - \lambda I$ is invertible on a space $X$ precisely when the operator $R_g(\lambda):X \to X$ defined above is bounded on $X$. Of course, this implies the norm equivalence $\|f\|_X \simeq \|R_g(\lambda)f\|_X$ for $f \in X$. This simple but crucial observation has been fruitfully employed in the study of the spectrum in \cite{acs}, and later in \cite{alemanpelaezhardy} and \cite{inteqvolterra}, where connections have been established between the boundedness of $R_g(\lambda)$ and the theory of Muckenhoupt weights and B\'ekoll\'e-Bonami weights. These connections will be our principal tools when establishing the main results of this paper.

In the Hardy space setting the boundedness of the operator $R_g(\lambda)$ is equivalent to a certain weight satisfying a Muckenhoupt $\A_\infty$-condition. A weight for us will be a positive measurable function, and the $\mathcal{A}_\infty$-class will consist of weights $w$ on $\T$ for which there exists a constant $C > 0$ such that the following estimate holds for all arcs $I \subseteq \T$:
\begin{equation} \label{Ainftycond}
\frac{1}{m(I)} \int_I w \, d\m \leq C \exp\Bigg( \frac{1}{m(I)}\int_I \log w \,d\m\Bigg).
\end{equation} Here $m(I)$ denotes the normalized Lebesgue measure of the interval $I$. We will also need to consider the Muckenhoupt $\A_2$-class, i.e the weights satisfying %For $1 < p < \infty$, the class $\A_p$ consists of weights satisfying the condition \begin{equation} \label{Apcond}
%\Bigg(\frac{1}{m(I)} \int_I w \,d\m\Bigg) \Bigg( \frac{1}{m(I)}\int_I w^{-\frac{1}{p-1}} \,d\m \Bigg)^{p-1} \leq C.
%\end{equation} 
\begin{equation} \label{Apcond}
\Bigg(\frac{1}{m(I)} \int_I w \,d\m\Bigg) \Bigg( \frac{1}{m(I)}\int_I w^{-1} \,d\m \Bigg) \leq C.
\end{equation}
It is well-known that $w\in \A_{2}$ if and only if $w, \, w^{-1}\in \A_\infty$. Moreover, $\A_\infty$ is closed under log-convex combination, in the sense that whenever $w_1, w_2 \in \A_\infty$ and $0 < r < 1$, then $w_1^rw_2^{1-r} \in \A_\infty$. 

The following result by Aleman and Pel\'aez characterizes the resolvent set $\rho(T_g | H^p)$ in terms of \eqref{Ainftycond}. 
\begin{thm}[Theorem C of \cite{alemanpelaezhardy}] \thlabel{Ainftyresolventcond} Assume that $\lambda \in \C \setminus \{0\}$, $0 < p < \infty$ and $g \in \BMOA$. Then, the following assertions are equivalent:

\begin{enumerate}[(i)]
\item $\lambda \in \rho(T_g|H^p)$,
\item $e^{g/\lambda} \in H^p$ and the weight $\exp(p \Re(g(e^{it})/\lambda))$ satisfies the $\A_\infty$-condition.
\end{enumerate}
\end{thm} 
A similar characterization has been obtained by Aleman and Constantin in the Bergman space setting, and was further refined by Aleman, Pott and Reguera in \cite{binfinity}. There the B\'ekoll\'e-Bonami weights appear instead, which are the Bergman space counterparts of the Muckenhoupt weights. The B\'ekoll\'e-Bonami class $B_2$ consists of weights on $\D$ satisfying an analogue of $\eqref{Apcond}$:

\begin{equation} \label{Bpcond}
\sup_{I\subset \mathbb{T}} \Bigg(\frac{1}{A(S_{I})} \int_{S_{I}} w \,dA \Bigg) \Bigg( \frac{1}{A(S_{I})}\int_{S_{I}} w^{-1} \,dA \Bigg) \leq C,
\end{equation} where $S_{I}$ denotes the usual Carleson square associated to the arc $I\subseteq \mathbb{T}$ with midpoint $\z_{I}$:
\begin{equation*}
S_{I} = \{ z \in \D : 1-m(I) < |z|, |\arg z - \arg \z_{I}| < m(I)/2 \},
\end{equation*} and $A(S_{I})$ is the normalized area measure of $S_{I}$. There is an analogue of the $\A_\infty$-class for the B\'ekoll\'e-Bonami weights which is often denoted by $B_\infty$ (see \cite{binfinity} for a precise definition of this class). The weights appearing in our context will be of the form $w(z) = |e^{g(z)}|$ with $g\in \B$, and in the case of such weights the results of \cite{binfinity} show that the $B_\infty$-class can be characterized in a similar way to \eqref{Ainftycond}: 

\begin{equation} \label{Binftycond}
\frac{1}{A(S_{I})} \int_{S_{I}} w \, dA \leq C \exp\Bigg( \frac{1}{A(S_{I})}\int_{S_{I}} \log w \,dA\Bigg).
\end{equation}

Similarly to the Muckenhoupt weights we have that $w, w^{-1} \in B_\infty$ if and only if $w \in B_2$, and $B_\infty$ is closed under log-convex combinations.

The condition that appears in the characterization of the spectrum of $T_g$ operators acting on the weighted Bergman spaces is the following. 

\begin{thm}[Part of Theorem A of \cite{acs}, Corollary 4.6 of \cite{binfinity}] \thlabel{B1resolventcond} Let $p > 0$ and $\alpha > -1$. For $\lambda \in \mathbb{C} \setminus \{0\}$ the following are equivalent:

\begin{enumerate}[(i)]
\item $\lambda \in \rho(T_g | L^{p,\alpha}_a),$
\item $e^{g/\lambda} \in L^{p,\alpha}_a$ and the weight $(1-|z|^2)^\alpha\exp(p \Re(g(z)/\lambda))$ satisfies the $B_\infty$-condition.
\end{enumerate}
\end{thm}

We will use the form of the resolvent in \eqref{resolventeq} and the conditions of \thref{Ainftyresolventcond} and \thref{B1resolventcond}  in the proofs of our main results. The results are stated in Section \ref{resultsection}, together with a discussion. The proofs are deferred to Section \ref{proofsection}.

\section{Main results} \label{resultsection}

Our first main result is a spectral stability property. This is a version of a result which appears in context of so-called growth classes in \cite[Theorem 5.4]{inteqvolterra}.

\begin{thmletter} \thlabel{specstab}
Let $g, h$ be analytic functions such that $T_g, T_h: X \to X$ are bounded, where $X = H^p$ or $X = L^{p,\alpha}_a$, $0 < p < \infty$, and such that the spectrum $\sigma(T_h | X)$ equals $\{0\}$. Then $$\sigma(T_{g+h} | X) = \sigma(T_g | X).$$
\end{thmletter}

The above property has been previously observed in some special cases, for instance when $g'$ is a rational function and $h$ induces a compact $T_h$ operator. See, for instance, \cite[Theorem 5.2]{youngspectra} or \cite[Theorem B]{acs}.

Our next result extends the applicability of \thref{specstab} by identifying a large class of quasi-nilpotent $T_g$ operators. The result holds true in a much larger class of spaces than just the Hardy and Bergman spaces. 

\begin{thmletter} \thlabel{zerospec} Let $X$ be a Banach space of analytic functions in $\D$ which contains the constants and such that the algebra $B(X)$ of bounded linear operators on $X$ contains the multiplication operators $M_h$ and the generalized Ces\`aro operators $T_h$ whenever $h \in H^\infty$. Then we have that $\sigma(T_g|X) = \{0\}$ whenever $T_g$ lies in the norm-closure of $\{T_h : h \in H^\infty \}$ in $B(X)$. 
\end{thmletter}

The conclusion of \thref{zerospec} also holds in the case of the metric spaces $H^p$ and $L^{p,\alpha}_a$ for $p \in (0,1)$, as will be clear from the proof given in Section \ref{proofsection}. The result is particularly useful in case that the space of symbols inducing bounded operators is known, as is the case for the Hardy and Bergman spaces. 
For instance, the following consequence is immediate from the well-known norm comparabilities $\|T_g\|_{H^p} \simeq \|g\|_{\BMOA}$ and $\|T_g\|_{L^{p,\alpha}_a} \simeq \|g\|_{\B}$.

\begin{corletter} \thlabel{zerospeccor} If $g$ lies in the norm-closure of $H^\infty$ in $\BMOA$ or in the norm-closure of $H^\infty$ in $\B$, then we have that $\sigma(T_g | H^p) = \{0\}$ and that $\sigma(T_g | L^{p,\alpha}_a )=\{0\}$, respectively.
\end{corletter}

It is natural to ask what can be said about the converse statement. If $g \in \BMOA$ or $g \in \B$, and the operator $T_g$ is quasi-nilpotent on $H^p$ or on $L^{p,\alpha}_a$, is then $g$ necessarily contained in the closure of $H^\infty$ in $\BMOA$, or $\B$, respectively? The Bergman case is related to a long-standing open problem which will be discussed below. In the case of the Hardy spaces the converse does hold. In fact we prove a stronger statement, with a nice geometric flavour. 

\begin{thmletter} \thlabel{hinfclosurebmoa} Let $0<p<\infty$ and $g\in \BMOA$. If the spectrum $\sigma( T_{g} \, \lvert \, H^{p} )$ does not contain any non-zero points of the real and imaginary axes, then $g$ lies in the norm-closure of $H^{\infty}$ in $\BMOA$, and thus $\sigma( T_{g} \, \lvert \, H^{p} ) = \{0\}$.
\end{thmletter}

In the statement of \thref{hinfclosurebmoa} the real and imaginary axes can be replaced by two arbitrary lines which intersect orthogonally at the origin. To the authors' best knowledge, all the explicit computations of spectra of $T_g$ on the Hardy spaces, for particular symbols $g$, reveal that if the spectrum is bigger than $\{0\}$, then it contains a disk with the origin on its boundary (see, for instance, \cite{ampcesaro}). This obviously implies that there are non-zero points in the spectrum which lie on either the real or imaginary axis.

\begin{figure} 
\centering
\begin{tikzpicture}[scale=1]
\begin{scope}
\pgfsetfillpattern{dots}{gray}

\filldraw[fill=gray!20] (-10,-6) to[out=20,in=230] (0,0) to[out=190,in=40] (-10,-6);
\end{scope}

\filldraw[black] (-10,-6) circle (2pt) node[anchor=east] {0};
\filldraw[black] (0,0) circle (2pt) node[anchor=west] {$\lambda$};
\filldraw[black] (-2,-6*1/5) circle (2pt) node[anchor=west] {$r_1\lambda$};
\filldraw[black] (-5,-6*1/2) circle (2pt) node[anchor=west] {$r_2\lambda$};
\filldraw[black] (-8,-6*4/5) circle (2pt) node[anchor=west] {$r_3\lambda$};
\draw [thin,domain=40.9:20] plot ({(4 + 36/25)^(1/2)*cos(\x) -10}, {(4 + 36/25)^(1/2)*sin(\x) -6}) node[anchor=north west] {$I_{r_3,\lambda}$};
\draw [thin,domain=39.5:22.3] plot ({(25 + 36/4)^(1/2)*cos(\x) -10}, {(25 + 36/4)^(1/2)*sin(\x) -6}) node[anchor=north west] {$I_{r_2,\lambda}$};
\draw [thin,domain=35.2:27] plot ({(64 + 36*16/25)^(1/2)*cos(\x) -10}, {(64 + 36*16/25)^(1/2)*sin(\x) -6}) node[anchor=north west] {$I_{r_1,\lambda}$};
\draw [->] (-10,-6) -- (-10, 2) node[above] {$\Im z$};
\draw [->] (-10,-6) -- (2, -6) node[right] {$\Re z$};
\end{tikzpicture}
\caption{Points in the shaded area belong to the spectrum.} \label{fig:fig1}
\end{figure}
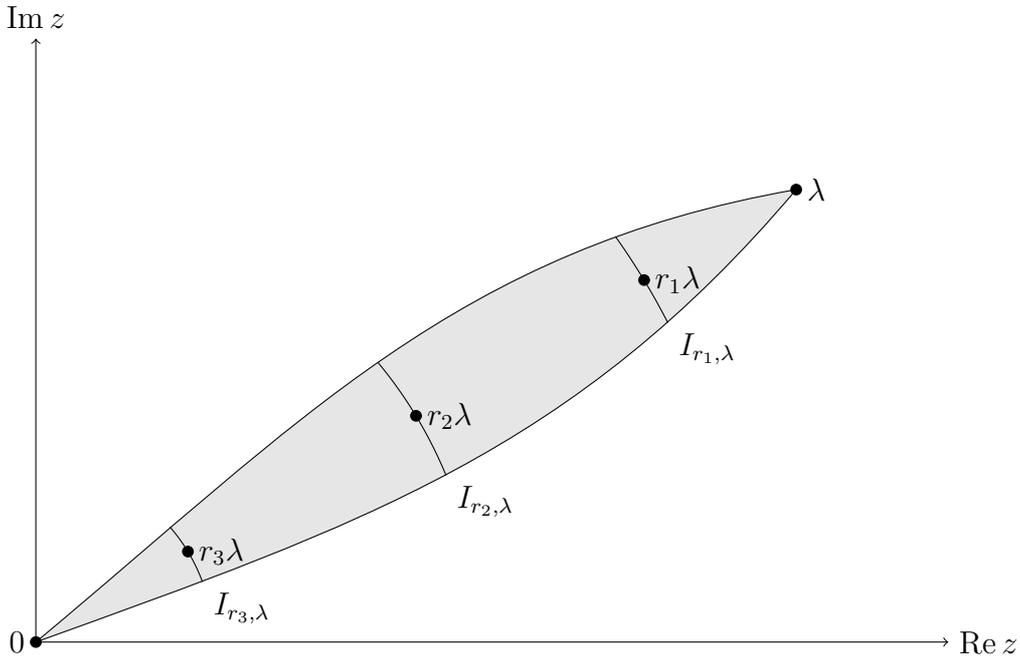

Our last main result provides further geometric properties of the spectrum. \thref{Ainftyresolventcond} and \thref{B1resolventcond} provide a deep link between the structure of the spectrum and the theory of exponential weights. It turns out that the convexity properties of the weight classes can be transferred to convexity properties of the spectrum.

\begin{thmletter}\thlabel{Geospec} Let $X= H^{p}$ or $X=L^{p,\alpha}_{a}$ for some $0<p< \infty$, and $T_g:X \to X$ be bounded. For every non-zero $\lambda \in \sigma (T_{g} \lvert X)$ and every $0< r <1$ there exists a circular arc $I_{r, \lambda}$ centered at $r\lambda$ (see Figure~\ref{fig:fig1}), such that the circular sector $S_{r,\lambda}$, created by taking the convex hull of the origin and $I_{r,\lambda}$ is contained in $\sigma(T_{g} \lvert  X)$. 
\end{thmletter}

The proof of \thref{hinfclosurebmoa} ultimately relies on distance formulas of Jones and Garnett which are a consequence of the well-known equivalence of the $\A_2$-condition and the Helson-Szeg\"o condition (see Chapter VI of \cite{garnett}). A similar equivalence is not known in the case of B\'ekoll\'e-Bonami class $B_2$. In particular the missing link is the rightmost inequality of $(\ref{Gar})$ and consequently the problem of establishing the converse of \thref{zerospeccor} in the case of the Bergman spaces remains unsolved. 

\begin{Conj*} Let $g \in \B$ be such that the weights $$w_\lambda(z) := \exp( \Re(g(z)/\lambda))$$ satisfy \eqref{Binftycond} for all $\lambda \in \mathbb{C} \setminus \{0\}$. Then $g$ lies in the closure of $H^\infty$ in the Bloch norm.
\end{Conj*}

With the characterization of \thref{B1resolventcond} in mind, the conjecture asserts that the quasi-nilpotency of $T_g$ on the Bergman spaces is equivalent to $g$ lying in the closure of $H^\infty$ in the Bloch space. The problem of characterizing this closure has been first stated in \cite{pommerenke} and remains open to this date. However, there exist results by Gal\'an and Nicolau \cite{galan2011closure} where a characterization of the closure of the $H^{p}$-spaces in the Bloch space is obtained in terms of square-type functions. This characterization does unfortunately not extend to $H^{\infty}$, as was shown by a counterexample in \cite{galanopoulos2015closure}.

\section{Proofs} \label{proofsection}

\begin{proof}[Proof of \thref{specstab}] We treat the case when $T_g$ and $T_h$ act on the $H^p$-spaces. We shall prove the inclusion $\rho(T_{g+h} | H^p) \subseteq \rho(T_g | H^p)$, from which the claim follows by considering $g+h$ and $-h$ instead of $g$ and $h$. 

Fix $\lambda \in \rho(T_g | H^p)$. Then by \thref{Ainftyresolventcond} the weight $\exp(p \Re(g(e^{it})/\lambda))$ satisfies the condition \eqref{Ainftycond}, and since the set $\rho(T_g | H^p)$ is open, the same holds for the weight $$w_1(e^{it}) = \exp(pp'\Re(g(e^{it})/\lambda))$$ whenever $p' > 1$ sufficiently close to 1. On the other hand, the weight $$w_2(e^{it}) = \exp(pq'\Re(h(e^{it})/\lambda))$$ satisfies the condition \eqref{Ainftycond} for all real $q'$ by \thref{Ainftyresolventcond} and the assumption that $\sigma(T_h | H^p) = \{0\}$. Let us now fix $p', q' > 1$ such that $1/p' + 1/q' = 1$ and so that $w_1$ satisfies \eqref{Ainftycond}. Let $$w(e^{it}) =  \exp(p \Re((g(e^{it})+h(e^{it}))/\lambda)).$$ Note that $w = w_1^{1/p'}w_2^{1/q'}$.  To show that $\lambda \in \rho(T_{g+h} | H^p)$ we now verify \eqref{Ainftycond} for $w$:

\begin{gather*}
\frac{1}{m(I)} \int_I w \, d\m \leq \Big(\frac{1}{m(I)} \int_I w_1 \, d\m\Big)^{1/p'}\Big(\frac{1}{m(I)} \int_I w_2 \, d\m\Big)^{1/q'} \\ \leq C \exp\Big( \frac{1}{m(I)}\int_I (1/p')\log w_1 \,d\m\Big) \cdot \exp\Big( \frac{1}{m(I)}\int_I (1/q')\log w_2 \,d\m\Big) \\ = C \exp\Big( \frac{1}{m(I)}\int_I (1/p') \log w_1 + (1/q') \log w_2 \,d\m\Big) \\ = C \exp\Big( \frac{1}{m(I)}\int_I \log w \,d\m\Big).
\end{gather*}

We need also verify that $e^{(g+h)/\lambda} \in H^p$. With the above choices of $p'$ and $q'$, we have that $e^{p'g/\lambda} \in H^p$ and $e^{q'h/\lambda} \in H^p$. The claim that $e^{(g+h)/\lambda} \in H^p$ now follows in a similar fashion from H\"older's inequality: \begin{gather*}\int_\T |e^{(g+h)/\lambda}|^p d\m \leq \Big(\int_\T |e^{p'g/\lambda}|^p \,d\m \Big)^{1/p'}\Big(\int_\T |e^{q'h/\lambda}|^p \,d\m \Big)^{1/q'} < \infty.
\end{gather*}

This proves the theorem for the case $X = H^p$. The case $X = L^{p,\alpha}_a$ is treated in an analogous way, by using instead the characterization of \thref{B1resolventcond}, the openness of the resolvent set and H\"older's inequality. We leave out the details of the computation which are similar to the above one.
\end{proof}

\begin{proof}[Proof of \thref{zerospec}] Fix $\lambda \in \mathbb{C} \setminus \{0\}$. Pick $h \in H^\infty$ with $\|T_g - T_h\|_{B(X)} = \|T_{g-h}\|_{B(X)}$ small enough for the operator $T_{g-h} - \lambda I$ to be invertible. For the constant function $1 \in X$ we have $$R_g(\lambda)1 = e^{\frac{h}{\lambda}}e^{\frac{g-h}{\lambda}} = M_{e^{h/\lambda}}R_{g-h}(\lambda)1 \in X.$$ On the other hand, if $f \in X$ with $f(0) = 0$, then 
\begin{gather*} R_g(\lambda)f(z) = e^{\frac{g(z)}{\lambda}} \int_0^z e^{-\frac{g(\z)}{\lambda}} f'(\z) \,d\z \\ = 
e^{\frac{h(z)}{\lambda}}e^{\frac{g(z)-h(z)}{\lambda}} \int_0^z e^{-\frac{g(\z)-h(\z)}{\lambda}} e^{-\frac{h(\z)}{\lambda}}f'(\z) \,d\z \\ = e^{\frac{h(z)}{\lambda}}e^{\frac{g(z)-h(z)}{\lambda}} \int_0^z e^{-\frac{g(\z)-h(\z)}{\lambda}} \Big[(e^{-\frac{h(\z)}{\lambda}}f(\z))' - e^{-\frac{h(\z)}{\lambda}}\frac{h'(\z)}{\lambda}f(\z) \Big] \,d\z\\ = M_{e^{h/\lambda}}R_{g-h}(\lambda)M_{e^{-h/\lambda}}f(z) + M_{e^{h/\lambda}}R_{g-h}(\lambda) T_{e^{-h/\lambda}}f(z).
\end{gather*}
Note that since $h \in H^\infty$, also $e^{\pm h/\lambda} \in H^\infty$ and thus the last line above is a sum of compositions of bounded operators on $X$. The operator $R_g(\lambda)$ is therefore itself bounded on $X$.
\end{proof}

In order to prove \thref{hinfclosurebmoa}, we will need a lemma which establishes the comparability of the distance to $H^{\infty}$ of a $\BMOA$-function in $\BMO$-norm to the distance to $L^{\infty}$ in $\BMO$-norm. According to \cite{garnett}, such a result seems to originate back to the work of D. Sarason in  \cite{zhu2007operator}. However, we have not been able to find an explicit proof of this result. For the sake of being self-contained, we include a short proof.

\begin{lemma} \thlabel{Sarason}There exists $C>0$, such that whenever $g\in \BMOA$,  
\begin{equation} \inf_{f \in H^{\infty}} \norm{g-f}_{\BMO} \leq C \inf_{h \in L^{\infty}} \norm{g-h}_{\BMO}
\end{equation}
\end{lemma}
\begin{proof} Without loss of generality, we may assume that $g(0)=0$. Now let $\left\{h_{n}\right\}_{n\in \mathbb{N}}\subset L^{\infty}$ be a sequence which minimizes the $\BMOA$-distance from $g$ to $L^{\infty}$.
%%Let $\varepsilon>0$ and choose $h_{\varepsilon}\in L^{\infty}$, such that
%%\begin{equation*} \norm{g-h_{\varepsilon}}_{\BMO} \, \leq   \inf_{h\in L^{\infty}}   \norm{g-h}_{\BMO} + \varepsilon
%%\end{equation*}
By simple triangle inequality, we have for all $n\in \mathbb{N}$
\begin{equation}\label{initial}\inf_{f\in H^{\infty}}   \norm{g-f}_{\BMO}   \leq     \norm{g-h_{n}}_{\BMO} + \inf_{f\in H^{\infty}}   \norm{h_{n} -f}_{\BMO} 
\end{equation}
%\begin{equation}\label{initial}
%\inf_{h\in H^{\infty}}   \norm{g-h}_{\BMO}   \leq   \inf_{h\in L^{\infty}}   \norm{g-h}_{\BMO} + \inf_{h\in H^{\infty}}   \norm{h_{\varepsilon} -h}_{\BMO} + \varepsilon
%\end{equation}
To estimate the second norm, we use the straightforward continuous embedding $L^{\infty} \hookrightarrow \BMO$ and the relation $\left(H^{1}_{0}\right)^{\perp} = H^{\infty}$ to get 
\begin{equation}\label{xtremal} \inf_{f\in H^{\infty}} \, \norm{h_{n}-f}_{\BMO}   \leq 2   \inf_{f\in H^{\infty}}   \norm{h_{n}-f} _{\infty}   =   2   \sup_{\substack{f\in H^{1}_{0} \\ \norm{f}_{H^{1}} \leq 1} }   \abs{ \int  f    h_{n}   dm}
\end{equation}
Since the Szeg\"o projection $P$ is self-adjoint and $P(\overline{g})=0$, we can write 
%It is well-known that the set of functions $\left\{F= u+iH(u)  :   \norm{u}_{L^{1}}+\norm{H(u)}_{L^{1}} < \infty\right\}$ forms a dense subspace of $H^{1}$ with $ \norm{F}_{H^{1}} \sim \norm{u}_{L^{1}} +\norm{H(u)}_{L^{1}} $. Applying Fefferman's duality theorem and recalling that the Hilbert transform is anti-symmetric, we get
\begin{equation}\label{cpair} \int_{\mathbb{T}}   f   h_{n}   dm  = \lim_{r\rightarrow 1-} \int_{\mathbb{T}} P(f)_{r} h_{n} dm = \lim_{r\rightarrow 1-}   \int_{\mathbb{T}}   f   \overline{P( \overline{h_{n}-g})_{r} }   dm  %\qquad F\in H^{1}_{0}  %\leq   c   \norm{F}_{H^{1}}   \norm{P\left(\overline{h_{n} -g}\right)  }_{\BMO}
\end{equation}
Now using duality on $(\ref{cpair})$ and the fact that $P: \BMO \rightarrow \BMOA$, we obtain
\begin{equation*}\sup_{\substack{f\in H^{1}_{0} \\ \norm{f}_{H^{1}} \leq 1} }   \abs{ \int  f    h_{n}   dm} \leq c \norm{P(\overline{h_{n}-g})}_{\BMO} \leq c'\norm{h_{n}-g}_{\BMO} 
\end{equation*}
%Since $g\in \BMOA$, we have $g = i   H(g)$, thus using the fact that $H:\BMO \rightarrow \BMO$ and applying the triangle inequality, we obtain 
%\begin{equation*} \norm{h_{n} - i H   (h_{n})}_{\BMO}   \leq    \norm{h_{n} -g}_{\BMO}   +   \norm{ H   (g-h_{n})}_{\BMO}   \leq   c'    \norm{h_{n} -g}_{\BMO} 
%\end{equation*}
Hence according to $(\ref{xtremal})$, we have established that $\inf_{f\in H^{\infty}}   \norm{h_{n}-f}_{\BMO}   \lesssim    \norm{h_{n} -g}_{\BMO}$, thus going back to $(\ref{initial})$, we ultimately arrive at
\begin{equation*} \inf_{f\in H^{\infty}}   \norm{g-f}_{\BMO}   \leq   C   \norm{g-h_{n}}_{\BMO} 
\end{equation*}
Letting $n\rightarrow \infty$, finishes the proof of the lemma.
\end{proof}
\begin{proof}[Proof of \thref{hinfclosurebmoa}]
Fix $0 < p < \infty$ and suppose we have an analytic function $g\in \BMOA$, with the property that $\sigma (T_{g} \lvert H^{p} ) \cap \left( \mathbb{R} \cup i  \mathbb{R} \right) = \left\{0 \right\}$. Set $w_{1} := \exp ( \Re(g) )$ and $w_{2} := \exp( \Im(g) ) $. Now since $\lambda, i\lambda \in \rho(T_{g} \lvert H^{p} ) $ for all $\lambda \in \mathbb{R} \setminus\left\{0\right\}$,  \thref{Ainftyresolventcond} yields that the weights $w_{1}^{1/\lambda}, w_{2}^{1/\lambda}$ both satisfy the $\A_{\infty}$-condition for all $\lambda \in \mathbb{R} \setminus\left\{0\right\}$. In particular, since both the weights $w_{j}^{1/\lambda}, w_{j}^{-1/\lambda}$ satisfy the $\A_{\infty}$-condition for $j=1,2$, the weights $w_{1}^{1/\lambda}, w_{2}^{1/\lambda}$ are in fact $\A_2$-weights, for all $\lambda >0$.
%\begin{gather*}\label{AinftoA2} \left(\frac{1}{m(I)} \int_{I} w^{1/\lambda}_{j}  dm\right)  \left(\frac{1}{m(I)}  \int_{I} w_{j}^{-1/\lambda}  dm\right) \\ \lesssim  \exp \left( \frac{1}{m(I)} \int_{I}  \log(w_{j}^{1/\lambda}) \right) \exp\left( \frac{1}{m(I)}  \int_{I} \log(w_{j}^{-1/\lambda}) dm \right)\\ = \exp \left(\frac{1}{m(I)} \int_{I} \log(w_{j}^{1/\lambda}) + \log(w_{j}^{-1/\lambda})  dm \right) = 1, \qquad \forall I \subset \mathbb{T} \text{ and }  j=1,2.
%\end{gather*} 
A well-known consequence of the Helson-Szeg\"o theorem (see \cite[Chapter VI, Section 6]{garnett}) is that there exists a constant $c>0$ such that for all real-valued $\phi\in \BMO$ we have
\begin{equation}\label{Gar}\frac{1}{c} \, \lambda(\phi) \leq \inf_{h\in L^{\infty} } \norm{\phi-h}_{\BMO} \leq c \lambda(\phi)
\end{equation}
where $\lambda(\phi) = \inf \left\{ \lambda >0  :  e^{\phi/\lambda} \in \A_{2} \right\}$.
Applying this fact to $\Re(g), \Im(g)\in \BMO$, we get that
$\inf_{h\in L^{\infty} } \norm{\Re(g)-h}_{\BMO} = \inf_{h\in L^{\infty} }  \norm{\Im(g)-h}_{\BMO} = 0$, thus $\inf_{h \in L^\infty} \norm{g - h}_{\BMO} = 0$. According to \thref{Sarason}, this is enough to conclude the proof. 

%the fact that the distance in $\BMO$-norm of an $\BMOA$-function to $H^{\infty}$ is comparable to its distance to $L^{\infty}$ in $\BMO$-norm, completes the proof of Theorem C.% and since $g \in \BMOA$, this can be improved to $\inf_{h \in H^\infty} \norm{g-h} = 0$ (see \cite[Exercise 22, Chapter VI]{garnett}).
\end{proof}
%To finish off the proof, we would need for the distance of an $\BMOA$-function to $H^{\infty}$ in $\BMO$-norm to be comparable to the distance to $L^{\infty}$ in $\BMO$-norm. 
%\end{proof}
%

%\noindent Indeed, \thref{Sarason} together with  $(\ref{Linf})$ and the sub-additivity of the distance is enough to conclude the proof of Theorem C, since 
%\begin{equation*} \inf_{h\in H^{\infty}}   \norm{g-h}_{\BMO}   \leq   C   \left(\inf_{h\in L^{\infty}}   \norm{ \Re(g) -h}_{\BMO} + \inf_{h\in L^{\infty}}   \norm{\Im(g) -h}_{\BMO} \right) = 0.
%\end{equation*}

%\begin{remark} Comment about the rightmost inequality of the Helson-Szeg\"o condition (and refer to it) that it is missing in the context of Bergman spaces.
%\end{remark}

Before moving on to the proof of \thref{Geospec}, we establish a preliminary lemma.

\begin{lemma}\label{starprop}
Let $X = H^{p}$ or $X=L^{p,\alpha}_{a}$ for some $0<p< \infty$. Then the spectrum $\sigma( T_{g} \lvert X )$ is star-shaped with respect to the origin, that is, if $\lambda \in \sigma( T_{g} \lvert X )$, then $r\lambda \in \sigma( T_{g} \lvert X )$ for each $r \in [0,1]$. In particular, the spectrum is always simply connected. 
\end{lemma}
\begin{proof} \sloppy  Without loss of generality we may assume that $\sigma (T_{g} \lvert  X  )\setminus\left\{0\right\}$ is non-trivial. Now suppose that there exists $0<r<1$ and a non-zero point $\lambda_{0} \in \sigma (T_{g}  \lvert  X )$, such that $r\lambda_{0} \in \rho ( T_{g} \lvert X )$. We will show that this implies that $\lambda_0$ is not in the spectrum, and thus the claim of the lemma will be established. %In a similar fashion as before, we want to use the characterizations of the spectrum of $T_{g}$ on the different spaces $X$ to arrive at a contradiction and 
We shall carry out the proof in the case $X=L^{p,\alpha}_{a}$. According to \thref{B1resolventcond} it suffices to establish that $e^{g/\lambda_{0}}\in L^{p,\alpha}_{a}$ and the weight $w_{\alpha}(z) = v_{\alpha}(z) \exp\left(p\Re(g(z)/\lambda_{0}) \right)$ satisfies the $B_{\infty}$-condition, where $v_{\alpha}(z)=(1-\abs{z}^{2})^{\alpha}$. Since $e^{g/r\lambda_{0}}\in L^{p,\alpha}_{a}$ by the assumption $r\lambda_{0} \in \rho ( T_{g} \lvert  L^{p,\alpha}_{a}  )$, the first assertion follows immediately from H\"older's inequality. For the second assertion, we assumed that $v_{\alpha}   \exp \left(p \Re(g/r\lambda_{0}))\right) $ satisfies the $B_{\infty}$-condition, hence applying H\"older's inequality followed by the $B_{\infty}$-condition, we obtain
\begin{gather*} \frac{1}{A(S_{I})}   \int_{S_{I}} w_{\alpha}   dA  =\frac{1}{A(S_{I})}   \int_{S_{I}} v_{\alpha}^{(1-r)\alpha}  v_{\alpha}^{r\alpha}    \exp\left(p\Re(g/\lambda_{0}) \right)    dA 
\\  \leq    \left(\frac{1}{A(S_{I})}   \int_{S_{I}} v_{\alpha}    \exp\left(p\Re(g/r\lambda_{0}) \right)    dA \right)^{r}    \left(  \frac{1}{A(S_{I})}   \int_{S_{I}} v_{\alpha}   dA \right)^{1-r}  
\end{gather*}\begin{gather*} \lesssim    \exp \left(\frac{1}{A(S_{I})}   \int_{S_{I}}  \log(w_{\alpha})   dA \right) \times \\ 
\left[ \exp\left( \frac{1}{A(S_{I})}   \int_{S_{I}}  \log(v^{-1}_{\alpha})   dA \right)     \frac{1}{A(S_{I})}   \int_{S_{I}} v_{\alpha}   dA  \right]^{1-r} \\ \lesssim \exp \left(\frac{1}{A(S_{I})}   \int_{S_{I}}  \log(w_{\alpha})   dA \right)
\end{gather*}
In the last step we also used the $B_{\infty}$-condition on the standard weights $v_{\alpha}$, which follows from the simple fact that the Bergman projection maps $P_{\alpha} : L^{p,\alpha}_{a} \rightarrow L^{p,\alpha}_{a}$. Hence by \thref{B1resolventcond}, we conclude that $\lambda_{0} \in \rho ( T_{g} \lvert L^{p,\alpha}_{a} )$, which contradicts our initial assumption, thus the theorem is proved in the case $X= L^{p, \alpha}_{a}$. The case $X=H^{p}$ is more straightforward and treated analogously, using instead the characterization of \thref{Ainftyresolventcond}.
\end{proof}

\begin{proof}[Proof of \thref{Geospec}]
This time, we carry out the proof in the case $X=H^{p}$, where the Bergman case is similar. Pick $\lambda \in \sigma(T_{g}  \lvert H^{p})\setminus \left\{0\right\}$. By means of multiplying $g$ with a unimodular constant, which corresponds to rotating the spectrum, we may without loss of generality assume that $\lambda>0$. Now fix $0<r<r'<1$ and consider the circular arc parametrization $\gamma_{r}(t) = r \lambda e^{it}$, $-\pi < t\leq \pi$. Notice that 
\begin{equation} \Re\left( \frac{g}{\gamma_{r}(t)} \right) = \frac{\Re(g)}{r\lambda} \cos(t) + \frac{\Im(g)}{r\lambda} \sin(t) = \frac{\Re(g)}{\frac{r}{\cos(t)}\cdot \lambda} + \frac{\Im(g)}{\frac{r\lambda}{\sin(t)}}
\end{equation}
Choose $t$ sufficiently small, such that $\frac{r}{\cos(t)} \leq r'$ and $\frac{r(1-r')\lambda}{r' \abs{\sin(t)}} > \abs{\sigma(T_{g}|H^p )}$ and denote this interval by $J_{r , \lambda}$. Then by \thref{Ainftyresolventcond} and the star-shaped property of the spectrum, we have that the weight 
\begin{equation}\label{uweight}u^{r'}_{t}= \exp\left(\frac{\Re(g)}{\frac{r}{r'  \cos(t)}\cdot \lambda} \right)\notin \mathcal{A}_{\infty} \qquad \forall t\in J_{r,\lambda}
\end{equation} while $v^{r'/(1-r')}_{t}=  \exp\left(\frac{\Im(g)}{\frac{r(1-r')\lambda}{r' \sin(t)}} \right)\in \mathcal{A}_{\infty}$. Since the assumptions are symmetric in $t$ and the sine-function is odd, we have that $v_{-t}= v_{t}^{-1}$, thus in fact, $v^{r'/(1-r')}_{t}$ is a Muckenhoupt $\mathcal{A}_{2}$-weight. Setting $w_{t} =  \exp\left(\Re\left( \frac{g}{\gamma_{r}(t)} \right)\right) $ we can rewrite $v_{t}^{-1}  w_{t} = u_{t}$. For the sake of obtaining a contradiction, suppose $w_{t_{0}}\in \mathcal{A}_{\infty}$, for some $t_{0}\in J_{r,\lambda}$, which by \thref{Ainftyresolventcond} corresponds to the assumption that $\gamma_{r}(t_{0}) \in \rho ( T_{g} \lvert H^{p} )$. Since the $\mathcal{A}_{\infty}$-class is closed under log-convex combinations, we get 
\begin{equation}u_{t_{0}}^{r'} = v_{t_{0}}^{-r'} w_{t_{0}}^{r'}= \left(\underbrace{v_{t_{0}}^{-r'/(1-r')}}_{\in \mathcal{A}_{\infty}}\right)^{1-r'} w_{t_{0}}^{r'} \in \mathcal{A}_{\infty}
\end{equation}
which contradicts $(\ref{uweight})$. We have established that whenever $t\in J_{r,\lambda}$, i.e. $$\abs{t} \leq  \min \left\{ \arccos(r/r'), \arcsin\left( r(1-r') \lambda/ r'  \abs{\sigma(T_{g} )} \right) \right\}$$ then the corresponding circular arc $I_{r,\lambda} \subseteq \sigma( T_{g} \lvert  H^{p} )$. Taking convex hulls of the origin and the circular arc $I_{r,\lambda}$, we conclude by Lemma $\ref{starprop}$ that the corresponding circular sector $S_{r,\lambda}$ is included in the spectrum. %(In order to demonstrate how the arc length changes in $0<r<1$, see attached figure).%Indeed, when $r$ increases from $0$ up to $1/2$, the arc $I_{r,\lambda}$ increases due to the fact that $\arccos$ increasing, while it decreases when $r$ goes from $1/2$ to $1$, since $\arcsin(1-r)/r)$ is decreasing. 
\end{proof}

\textbf{Acknowledgement.} The authors want to express their gratitude towards Alexandru Aleman, for our useful discussions and his ideas surrounding this work, in particular a suggestion for an elegant proof of \thref{zerospec}.

\bibliographystyle{siam}
\bibliography{mybib}

\end{document}